\documentclass[11pt]{amsproc}
 \usepackage[margin=1in]{geometry}
\usepackage{setspace,fullpage}
\usepackage{graphicx}
\usepackage[nice]{nicefrac}
\usepackage{amssymb}
\usepackage{amsmath,amsthm,amscd,amssymb, mathrsfs}

\usepackage{latexsym}
\usepackage[colorlinks,citecolor=red,pagebackref,hypertexnames=false]{hyperref}

\numberwithin{equation}{section}

\theoremstyle{plain}
\newtheorem{theorem}{Theorem}[section]
\newtheorem{lemma}[theorem]{Lemma}

\newtheorem{conjecture}[theorem]{Conjecture}

\theoremstyle{definition}

\theoremstyle{remark}
\newtheorem{remark}[theorem]{Remark}

\newtheorem{case[theorem]}{Case}

\title[\parbox{14cm}{\centering{ Sharp extension theorems and Falconer distance problems  \hspace{1in}}} \quad]{Sharp extension theorems and Falconer distance problems for algebraic curves in two dimensional vector spaces over finite fields }
\author{ Doowon Koh and Chun-Yen Shen }

\address{Department of Mathematics\\
Michigan State University \\
East Lansing, MI 48824,  USA}
\email{koh@math.msu.edu}

\address{Department of Applied Mathematics\\
 National Chiao Tung University\\
HsinChu 300, Taiwan}
\email{shenc@umail.iu.edu}

\thanks{ }


\begin{document}

\begin{abstract} In this paper we study extension theorems associated with general varieties in two dimensional vector spaces over finite fields. Applying Bezout's theorem, we obtain the sufficient and necessary conditions on general curves where sharp $L^p-L^r$ extension estimates hold. Our main result can be considered as a nice generalization of works by Mochenhaupt and Tao in \cite{MT04} and Iosevich and Koh in \cite{IK08}. As an application of our sharp extension estimates, we also study the Falconer distance problems in two dimensions.
\end{abstract}
\maketitle
\tableofcontents

\section{Introduction}
In the Euclidean setting, the extension theorem is one of the most important, challenging open problems in harmonic analysis.
Since this problem was first addressed in 1967 by Stein (\cite{St78}), it has been extensively studied in the last few decades, in part because it is closely related to other interesting problems in harmonic analysis, such as Kakeya problems, Falconer distance problems, and Bochner-Riesz summability problems.
In the Euclidean case, the extension theorem asks us to determine the optimal range of exponents $1\leq p,r \leq \infty$ such that the following extension estimate holds:
$$\|(fd\sigma)^\vee\|_{L^r({\mathbb R}^d)}\leq C(p,r,d)\|f\|_{L^p(V, d\sigma)} ~~\mbox{for all}~~ f\in L^p(V, d\sigma)$$
where $d\sigma$ is a measure on the set $V$ in ${\mathbb R}^d$ and $(fd\sigma)^\vee$ denotes the inverse Fourier transform of the measure $fd\sigma.$
In the case when the set $V$ is a hypersurface such as the sphere, the paraboloid, or the cone in ${\mathbb R^d}$, the extension problems have received much attention and they were completely solved in lower dimensions. For example, the complete solution for the circle or the parabola in ${\mathbb R^2}$ is due to Zygmund (\cite{Zy74}), and the complete solutions for the cone in ${\mathbb R^3}$ and the cone in ${\mathbb R^4}$ are due to Barcelo (\cite{Ba85}) and Wolff (\cite{Wo01}) respectively.
Currently, the best known results for the cones in ${\mathbb R^d}, d\ge 5$, and the spheres or the paraboloids in ${\mathbb R^d}, d\geq 3$, are due to Wolff (\cite{Wo01}) and Tao (\cite{Ta03}) respectively, in which they used bilinear approach. However, it has been believed that their results can be improved and new ideas seem to be needed to totally understand the extension problems. For a comprehensive survey of these problems in the Euclidean case, see \cite{Ta04} and the references therein. \\

In recent year the extension problems have been also investigated in the finite field setting. The finite field case serves as a typical model for the Euclidean case and it also possesses  structural advantages which enable us to  relate our problems to other well-studied problems in number theory, arithmetic combinatorics, or algebraic geometry. Therefore, we may find useful techniques from these fields to attack our problems.
Moreover, the finite field problems  display independently interesting features. For these reasons, Mochenhaupt and Tao (\cite{MT04}) first constructed the extension problems in the finite field setting and they provided us of remarkable facts related to extension theorems for several kinds of algebraic varieties. In particular, they gave us the complete solution for the parabola in two dimensional vector spaces over finite fields. It was Iosevich and Koh that continued the works by Mochenhaupt and Tao. In \cite{IK09}, they studied the $L^p-L^r$ boundedness of the extension operators associated with paraboloids in higher dimensions and partially improved the results by Mochenhaupt and Tao. In \cite{IK08}, they also studied general spherical extension problems and they especially obtained the complete solution for the extension problems related to non-degenerate quadratic curves in two dimension. \\

In this paper we investigate general properties of algebraic varieties in two dimensional vector space of finite fields on which the extension problems are completely understood. We have the following main theorem which is a generalization of the sharp extension theorems for the parabola in \cite{MT04} and the non-degenerate quadratic curves in \cite{IK08} respectively.

\begin{theorem}\label{Main} Let $\mathbb F_q$ denote a finite field with $q$ elements where we assume that $q$ is a power of odd prime.
Suppose that $P(x)\in {\mathbb F_q}[x_1,x_2]$ is non-zero polynomial.  Define an algebraic variety $V \subset \mathbb F_q^2$ by
$$V=\{x\in {\mathbb F_q^2}: P(x)=0\}.$$
Then, $L^2-L^4$ extension estimate related to $V$ holds if and only if
the polynomial $P(x)$ does not have any linear factor and $|V|\sim q.$
\end{theorem}
Here and throughout the paper, we denote by $|V|$ the cardinality of $V,$ and $|V|\sim q$ means that there exist $C,c>0$ depending only on the degree of the polynomial $P(x)$ such that $c q \leq |V|\leq Cq.$ We also notice that the norm of the extension operator depends only on the degree of $V$ and on the ratio $|V|/q.$ In addition, we assume that the characteristic of the underlying finite field $\mathbb F_q$ is greater than the degree of $V.$

\begin{remark} In section \ref{section2}, the definition of extension problems in finite fields is reviewed and we will give the proof of Theorem \ref{Main} in section \ref{section3}. The $L^2-L^4$ estimate in Theorem \ref{Main} gives the critical exponents for all possible exponents where the extension estimate holds (see Remark \ref{endpoint} and the necessary conditions (\ref{N1}) in section \ref{section2} ). This presents an interesting fact that there exist some differences between the finite field case and the Euclidean case. For example, let us consider a set $V$ consisting of all zeros of $x_1^4+x_2^4-1=0.$ In the Euclidean case, the extension estimate for this variety $V$ is much worse than that for the circle variety, $\{x\in \mathbb R^2:x_1^2+x_2^2=1\}$, because $V$ is a curve with a vanishing Gaussian curvature (see \cite{DI95} or the pages $414$ and $418$ in \cite{St93}). However, Theorem \ref{Main} says that the circle and the variety $V$ yield the same extension estimate in finite fields.  Another difference is that the curve in the finite field case yields much better extension estimate than its counterpart of the Euclidean case. For instance, the $L^4-L^4$ estimate gives the critical exponents up to the endpoints for the circular extension estimate in $\mathbb R^2$ (see \cite{Ta04}). However, this best possible result is much worse than the $L^2-L^4$ extension estimate which yields sharp exponents in finite fields case.

\end{remark}

Another interesting problem in harmonic analysis is the Falconer distance problem which is closely related to the extension problems. Let $E \subset {\bf R}^d$, $d \ge 2$ be a compact subset. In the Euclidean setting, the Falconer distance problem is to find $s_0>0$ such that if the Hausdorff dimension of $E$ is greater than $s_0$, then one-dimensional Lebesgue measure of $\Delta(E,E)$ is positive, where  $\Delta(E,E)$ denotes the distance set given by
$$ \Delta(E,E)=\{|x-y|\in {\bf R}: x, y\in E\}.$$
This problem was first addressed by Falconer (\cite{Fa85}) who conjectured that if the Hausdorff dimension of $E\subset R^d$ is greater than $d/2$, then the Lebesgue measure of $\Delta(E,E)$ is positive. This problem has not been solved in all dimensions.
Using the estimates of the Fourier transform of the characteristic function of an annulus in ${\bf R^d}$, Falconer in \cite{Fa85} first obtained the following nontrivial result:
$$ \mbox{if}\quad \mbox{dim}(E)> \frac{d+1}{2}, ~~\mbox{then}~~ |\Delta(E,E)|>0,$$
where $\mbox{dim}(E)$ denotes the Hausdorff dimension of $E\subset {\bf R^d}$ and $|\Delta(E,E)|$ denotes one-dimensional Lebesgue measure of the distance set $\Delta(E,E).$
In \cite{Ma87}, Mattila generalized the Falconer result by showing that
 $\mbox{if}\quad \mbox{dim}(E)+\mbox{dim}(F)> d+1, ~~\mbox{then}~~ |\Delta(E,F)|>0,$
where $E, F$ are compact subsets of ${\bf R^d}$ and $\Delta(E,F)=\{|x-y|\in {\bf R}: x\in E, y\in F\}.$ Moreover, he reduced the Falconer distance problem to estimating the spherical means of Fourier transforms of measures. Using Mattila's approach, Wolff (\cite{Wo99}) proved that in two dimension, $\mbox{dim}(E) > 4/3$ implies $|\Delta(E,E)|>0$ which is the best known result in two dimension. Applying Mattila's approach along with the weighted version of Tao's bilinear extension theorem (\cite{Ta03}), Erdo\~{g}an in \cite{Er05} obtained the best known results in higher dimensions: if $\mbox{dim}(E)> d/2+1/3$, then $|\Delta(E,E)| >0.$ \\

In \cite{IR07}, Iosevich and Rudnev first studied an analog of the Falconer distance problem in the finite field setting and Iosevich and Koh (\cite {IK07}) studied the problems related to the general cubic distances.
Let $E$ be a subset of  ${\mathbb F}_q^d$, $d \ge 2$, the $d$-dimensional vector space over the finite field $\mathbb F_q.$  For each $x\in \mathbb F_q^d $ and   a positive integer $n \geq 2$, we define  $\|x\|_n = x_1^n+\cdots + x_d^n.$
Let $ \Delta_n(E,E)=\{\|x-y\|_n: x,y \in E\}$  viewed as a subset of $ {\mathbb F}_q.$ The Falconer distance problem in this context asks for the smallest number $s_0$ such that $\Delta_n(E,E)$ contains a positive proportion of the elements of ${\mathbb F}_q$ provided that
$|E| \ge Cq^{s_0}$. Iosevich and Koh conjectured in \cite{IK07} that if $|E|\geq C q^{d/2}$ with $C$ sufficiently large, then  $|\Delta_n(E,E)| \gtrsim q$, which generalizes the conjecture originally stated in \cite{IR07} for $n=2$. However, it turned out that the conjecture is not true in the case when $n=2$ and  the dimension $d$ is odd. In fact, arithmetic examples constructed by the authors in \cite{HIKR10} show that the exponent $(d+1)/2$ is sharp. However, it has been believed that the conjecture may hold if the dimension $d$ is even, in part because the extension theorem for spheres in even dimensions can be better than in odd dimensions for finite fields. The authors in \cite{CEHIK09} recently showed that
if $E \subset {\mathbb F}_q^2$ of cardinality $\geq Cq^{4/3}$, with $C>0$ large, then we have
\begin{equation}\label{kohoriginal} |\Delta_2(E,E)| \gtrsim q.\end{equation}
When $d=2$, the exponent $4/3$ not only matches the one obtained by Wolff in reals, but gives a better result than the exponent $(d+1)/2$ which gives a sharp exponent in odd dimensions.  In this paper, we shall show that the general version also holds. Namely, we have the following main theorem for the Falconer distance problem in two dimension.

\begin{theorem}\label{sharpwolff1}
Let $E, F\subset {\mathbb F_q^2}.$
If $|E||F|\gtrsim q^{\frac{8}{3}}$, then we have
$$ |\Delta(E,F)|:=|\{\|x-y\|_2\in {\mathbb F_q}: x\in E, y\in F\}| \gtrsim q,$$
here, throughout the paper, we define $\|x\|_2=x_1^2+x_2^2.$
\end{theorem}
In particular, our result in two dimension improves the Shparlinski's work in \cite{Sh06} which says that
$$ |\Delta_2(E,F)| > \frac{|E||F|q}{q^{d+1}+|E||F|}.$$
Here, we recall that for positive numbers $X$ and $Y$, the notation $X\lesssim Y$ means that there exists a constant $C>0$ independent of the parameter $q$ such that $X\leq C Y.$ For a complex number $A$ and a non-negative real number $B$, the notation $A=O(B)$ is used if  $ |A|\leq CB$ for some $C>0$ independent of the size of the underlying finite field $\mathbb F_q.$

\begin{remark}\label{address} The proof of Theorem \ref{sharpwolff1} will be given in section \ref{section4}. To prove our main theorem, we mainly follow the methods which the authors in \cite{CEHIK09} used to obtain the result given in (\ref{kohoriginal}). However, our Lemma \ref{keylemma1} which plays a crucial role in the proof was obtained by a new method and it can be also applied to a more general setting. We hope that our Lemma \ref{keylemma1} and Lemma \ref{restriction} provide a clue to attack the generalized distance problems such as the cubic distance problem which was initially studied by the authors in \cite{IK07}.
\end{remark}


\section{Notation and definitions for extension problems} \label{section2}
We review some notation and definitions related to the extension problems in the finite field setting. We shall use the notation and definitions given in \cite{IK08} or \cite{IK09}.  Let ${\mathbb F}_q$ be a finite field with $q$ elements. We denote by ${\mathbb F}_q^d$  the $d$-dimensional vector space over the finite field ${\mathbb F}_q$. We shall work on the function space $({\mathbb F}_q^d, dx)$
where a normalized counting measure $dx$ is always endowed and our algebraic varieties shall be defined. Therefore,
given a complex valued function $f: ({\mathbb F}_q^d, dx) \to {\mathbb C}$,  the integral of the function $f$ over the function space $({\mathbb F}_q^d, dx)$ is given by
$$\int_{{\mathbb F}_q^d} f(x)~dx = \frac{1}{q^d} \sum_{x\in {\mathbb F_{q}^d}} f(x).$$
We also define the frequency space $({\mathbb F}_q^d, dm)$ as the dual space of the function space $({\mathbb F}_q^d, dx)$,
where we always endow the frequency space with the counting measure $dm.$ For a fixed non-trivial additive character $\chi: {\mathbb F}_q \to {\mathbb C},$ we therefore define the Fourier transform of the function $f$ on $({\mathbb F}_q^d, dx)$ by the following formula
\begin{equation}\label{defofrF} \widehat{f}(m)= \int_{{\mathbb F_q^d}} \chi(-m\cdot x)f(x)~dx= \frac{1}{q^d}\sum_{x\in {\mathbb F_q^d}} \chi(-m\cdot x) f(x),\end{equation}
where $m $ is an element of the dual space $({\mathbb F_q^d} ,dm ).$ Recall that the Fourier transform of the function $f$ on $({\mathbb F_q^d}, dx)$ is actually defined on the dual space $({\mathbb F_q^d}, dm).$ Here we endow the dual space $({\mathbb F_q^d}, dm)$ with a  counting measure $dm.$ We therefore see that the Fourier inversion theorem holds:
\begin{equation}\label{inversionF}f(x)=\int_{m\in {\mathbb F_q^d}} \chi(m\cdot x) \widehat{f}(m)~dm=\sum_{m\in {\mathbb F_q^d}} \chi(m\cdot x) \widehat{f}(m),\end{equation}
where $f$ is the complex valued function defined on $({\mathbb F_q^d}, dx).$  Using the orthogonality relation of the non-trivial additive character, meaning that $\sum_{x\in \mathbb F_q^d}\chi(m\cdot x)=0$ for $m\neq (0,\dots,0)$, we also see that  the Plancherel theorem holds:
$$ \|\widehat{f}\|_{L^2({\mathbb F}_q^d, dm)}=\|f\|_{L^2({\mathbb F_q^d},dx)}.$$
In other words, the Plancherel theorem takes the following formula:
\begin{equation}\label{plancherel} \sum_{m\in {\mathbb F_q^d}} |\widehat{f}(m)|^2= \frac{1}{q^d}\sum_{x\in {\mathbb F_q^d}} |f(x)|^2.\end{equation}
Let $f$ and $h$ be the complex valued functions defined on the function space $({\mathbb F_q^d}, dx).$
The convolution of $f$ and $g$ is defined on the function space $(\mathbb F_q^d, dx)$ by the formula
$$f\ast h (y)=\int_{x\in {\mathbb F_q^d}} f(y-x)g(x)~dx = \frac{1}{q^d}\sum_{x\in {\mathbb F_q^d}} f(y-x)~g(x).$$ Then, we can easily check that
$$ \widehat{(f\ast h)}(m)= \widehat{f}(m)\cdot \widehat{h}(m).$$
\begin{remark} Throughout the paper we always consider the variable $``x"$ as an element of the function space $({\mathbb F_q^d}, dx)$ with the normalized counting measure $dx$. On the other hand, we always use  the variable $``m"$ for the element of the frequency space $({\mathbb F_q^d}, dm)$ with the counting measure $dm$.
\end{remark}

\subsection{Extension problems for general algebraic varieties in ${\mathbb F_q^2}$ }
We now introduce algebraic varieties $V \subset ({\mathbb F_q^2}, dx)$ on which we shall work. In addition, we review the definition of extension problems related to the variety $V.$ Let $P(x)\in {\mathbb F_q}[x_1,x_2]$ be a polynomial with degree $k.$ Throughout the paper we always assume that the degree of the polynomial $P(x)$ is less than the characteristic of the underlying finite field ${\mathbb F_q}.$ We shall consider the following algebraic variety $V \subset ({\mathbb F_q^2},dx)$ generated by the polynomial $P(x)\in {\mathbb F_q}[x_1,x_2]$ :
$$V=\{x\in {\mathbb F_q^2}: P(x)=0\}.$$
We endow the variety $V$ with a normalized surface measure $d\sigma$ defined by the relation
\begin{equation}\label{measuredef} \int_{V} f(x) ~d\sigma(x) = \frac{1}{|V|}\sum_{x\in V} f(x),\end{equation}
where $|V|$ denotes the number of elements in $V.$ Note that the total mass of $V$ is one and the measure $\sigma$ is just a function on $({\mathbb F_q^2, dx})$ given by
\begin{equation}\label{measuredef1} d\sigma(x)= \frac{q^d}{|V|} V(x),\end{equation}
here, and throughout the paper, we identify the set $V$ with the characteristic function on the set $V.$ For instance, we write $E(x)$ for $\chi_{E}(x)$
where $E$ is a subset of the function space $({\mathbb F_q^2}, dx).$  For $1\leq p, r \leq \infty,$ we denote by $R^*(p\to r)$ the smallest constant such that the following extension estimate holds:
\begin{equation}\label{test} \|(fd\sigma)^\vee\|_{L^r({\mathbb F_q^2}, dm)} \leq R^*(p\to r) \|f\|_{L^p(V,d\sigma)}\end{equation}
for every function $f$ defined on $V$ in $({\mathbb F_q^2}, dx ),$ where the inverse Fourier transform of the measure $fd\sigma$ takes the form:
$$ (fd\sigma)^\vee(m)=\int_{V} \chi(m\cdot x) f(x)~d\sigma(x)= \frac{1}{|V|}\sum_{x\in V} \chi(m\cdot x) f(x).$$
By duality,  $R^*(p\to r)$ is also defined as the smallest constant such that the following restriction estimate holds:
\begin{equation}\label{defrestriction}
\|\widehat{g}\|_{L^{p^\prime}(V,d\sigma)} \leq R^*(p\to r) \|g\|_{L^{r^\prime}(\mathbb F_q^d, dm)}
\end{equation}
for all functions $g$ on $(\mathbb F_q^d,dm)$ where $p^\prime$ and $ r^\prime$ denote the dual exponents of $p$ and $r$ respectively which mean $1/p+1/p^\prime =1$ and $1/r+1/r^\prime =1.$ The constant $R^*(p\to r)$ may depend on $q,$ the size of the underlying finite field
${\mathbb F_q}.$ However, the extension problem asks us to determine the exponents $1\leq p,r\leq \infty$ such that $R^*(p\to r)\lesssim 1$ where the constant under the notation $\lesssim$ is independent of $q$ and it depends only on the degree of the variety $V$ and on the ratio $|V|/q.$

\begin{remark}\label{keydef} Here, we need to be careful with the definition of $\widehat{g}$ in the restriction estimate (\ref{defrestriction}). Sine $g$ is defined on $(\mathbb F_q,dm)$ with a counting measure $``dm"$, the Fourier transform $\widehat{g}$ is actually defined on the dual space $(\mathbb F_q^d,dx)$ with the normalized counting measure. Thus, the Fourier transform of the function $g$ takes the following : for each $x\in (\mathbb F_q^d, dx),$
\begin{equation}\label{anotherFourier} \widehat{g}(x)=\int_{\mathbb F_q^d} \chi(-x\cdot m) g(m)dm = \sum_{m\in \mathbb F_q^d} \chi(-x\cdot m) g(m),\end{equation}
which is different from the definition of the Fourier transform in (\ref{defofrF}). Notice that the difference happened because the Fourier transform depends on its domain. Thus, when we compute the Fourier transform, we must carefully check its domain. \end{remark}

\begin{remark}\label{endpoint}A direct calculation yields the trivial estimate, $R^*(1\to \infty)\lesssim 1.$
Using H\"{o}lder's inequality and the nesting properties of $L^p$-norms, we also see that
$$ R^*(p_1\to r)\leq R^*(p_2\to r)\quad\mbox{for}~~1\leq p_2\leq p_1\leq \infty$$
and
$$ R^*(p\to r_1)\leq R^*(p\to r_2)\quad \mbox{for}~~ 1\leq r_2\leq r_1\leq \infty.$$\\
For any fixed exponent $1\leq p ( \mbox{or}~r )\leq\infty$, we therefore aim to find the smallest exponent $1\leq r( \mbox{or}~p )\leq\infty$ such that $R^*(p\to r)\lesssim 1.$ By interpolating the result $R^*(p\to r)\lesssim 1 $ with the trivial bound $R^*(1\to \infty)\lesssim 1$, further results can be obtained.\end{remark}

\subsection{Necessary conditions for $R^*(p\to r)\lesssim 1$ }

Mochenhaupt and Tao in \cite{MT04} observed the necessary conditions for the boundedness of extension operators related to general algebraic varieties in $d$-dimensional vector spaces over finite fields. For example, if $V \subset({\mathbb F_q^2},dx)$ is an algebraic variety with $|V|\sim q^s$ for some $0<s<2$, then the necessary conditions for $R^*(p\to r)\lesssim 1$ take the following:
\begin{equation}\label{N1} r\geq \frac{4}{s}\quad \mbox{and} \quad  r\geq \frac{2p}{s(p-1)}.\end{equation}
However, if  $V$ contains a $\alpha$-dimensional affine subspace $H (|H|=q^{\alpha})$, the necessary conditions in (\ref{N1}) can be improved by adding the condition
\begin{equation}\label{N3}
r\geq \frac{p(2-\alpha)}{(p-1)(s-\alpha)}
\end{equation}
For the proof of above necessary conditions, see pages $41-42$ in \cite{MT04}.
\begin{remark}\label{linebad}Let $V=\{x\in {\mathbb F_q^2}: P(x)=0\}$ be an algebraic variety in $({\mathbb F_q^2}, dx),$ where $P(x)\in {\mathbb F_q}[x_1,x_2]$ is a non-zero polynomial. Then it is clear that $|V|\lesssim q$ where the constant depends only on the degree of the polynomial $P(x)$ and the variety $V$ can only contain zero or one dimensional affine subspace. If $V$ contains one dimensional affine subspace  $H$(a line) and $|V|\sim q$,  then the necessary condition (\ref{N3}) says that there are no extension estimates except the trivial cases, $R^*(p\to \infty) \lesssim 1$ for $1\leq p\leq \infty.$ Thus, we are only interested in the case when our variety $V$ does not contain any line. Namely, the polynomial $P(x)$ generating the variety $V$ does not have any linear factor. In this case, if $|V|\sim q,$ then the necessary conditions (\ref{N1}) exactly takes the following:
$$ r\geq 4\quad \mbox{and} \quad  r\geq \frac{2p}{p-1}.$$
In fact, the necessary condition, $r\geq \frac{2p}{p-1}$, can be obtained by testing (\ref{test}) with the function $f$ supported in one point of $V.$ If $V$ contains a large subset $H$ of a line, then this necessary condition must be improved by testing (\ref{test}) with the characteristic function on $H.$ However, Bezout's Theorem (see Theorem \ref{Bezout'sTheorem}) says that the cardinality of $H$ can not be more than the degree of $P(x)$, because otherwise $V$ must contain the line. This observation leads us to the following conjecture.\end{remark}

\begin{conjecture}\label{conj} Given a non-zero polynomial $P(x)\in {\mathbb F_q}[x_1,x_2],$ define an algebraic variety $V$ by
$$V=\{x\in {\mathbb F_q^2}: P(x)=0\}.$$
If $|V|\sim q$ and  the polynomial $P(x)$ does not have any linear factor, then the necessary conditions (\ref{N1}) are in fact sufficient conditions for $R^*(p\to r)\lesssim 1.$\end{conjecture}
In the case when $V=\{x\in \mathbb F_q^2: x_1^2-x_2=0\}$ is the parabola, Mochenhaupt and Tao in \cite{MT04} proved that Conjecture \ref{conj} holds.
Iosevich and Koh in \cite{IK08} also showed that Conjecture \ref{conj} is true if $V=\{x\in \mathbb F_q^2: a_1x_1^2+ a_2x_2^2=j\} $ with $a_1,a_2, j\ne 0$ is the nondegenerate quadratic curve. In fact, Theorem \ref{Main} shows that  Conjecture \ref{conj} is true for arbitrary algebraic curves.
To see this, notice from Remark \ref{endpoint} that if $|V|\sim q,$ then  $ R^*(2\to 4)\lesssim 1$ implies $R^*(p\to r)\lesssim 1$ for all exponents $(p,r)$ satisfying the necessary conditions (\ref{N1}).

\section{Proof of Theorem \ref{Main} }\label{section3}
In this section we shall restate and prove  our main theorem for extension problems.
The proof  is based on the following Bezout's Theorem (see \cite{Co59}) along with the method used to obtain the sharp extension estimates for the parabola in \cite{MT04} and nondegenerate quadratic curves in \cite{IK08}.
\begin{theorem} [Bezout's Theorem]\label{Bezout'sTheorem}
Two algebraic curves of degrees $K_1$ and $K_2$  can not intersect more than $K_1\cdot K_2$ points unless they have a component in common.
\end{theorem}
We also need  Schwartz-Zippel Lemma (see \cite{Zi79} and \cite{Sc80}).
\begin{lemma}[Schwartz-Zippel] Let $P(x)\in \mathbb F_q[x_1, x_2] $ be a non zero polynomial with
degree $k$. Then, we have
$$ |\{x\in \mathbb F_q^2: P(x)=0\}|\leq kq.$$
\end{lemma}

Now, we restate our main theorem for extension problems and give a complete proof.
\vskip 0.1in
\noindent{\bf Theorem \ref{Main}.} {\it
Suppose that $P(x)\in {\mathbb F_q}[x_1,x_2]$ is non-zero polynomial. Define an algebraic variety $V \subset \mathbb F_q^2$ by
$$V=\{x\in {\mathbb F_q^2}: P(x)=0\}.$$
Then, $R^*(2\to 4)\lesssim 1$ if and only if
$|V|\sim q$ and  the polynomial $P(x)$ does not have any linear factor.}

\begin{proof} \textbf{$(\Longrightarrow )$}  Suppose that $L^2-L^4$ extension estimate holds. By contradiction, assume that
$|V|$ is not $\sim q,$ or the polynomial $P(x)$ contains a linear factor.
If $|V|$ is not $\sim q$, then Schwartz-Zippel lemma says that $|V|\sim q^\varepsilon$ for some $0\leq \varepsilon <1.$
From the necessary condition (\ref{N1}), we therefore see that $L^2-L^4$ extension estimate is impossible.
On the other hand, if $P(x)$ has a linear factor, then the variety $V=\{x\in \mathbb F_q^2: P(x)=0\}$ contains a line.
In this case, $L^2-L^4$ extension estimate is also impossible, which is the immediate result from the necessary condition (\ref{N3}). \\
\textbf{ $( \Longleftarrow)$} Suppose that $|V|\sim q$ and  the polynomial $P(x)$ does not have any linear factor.
First, we prove the following key lemma.
\begin{lemma}\label{mainlemma} For each $j=1,2,\dots,n$, let $P_j(x)\in {\mathbb F_q}[x_1,x_2]$ be an irreducible polynomial with degree $\geq 2.$ Then, given each variety $V_j:=\{x\in {\mathbb F_q^2}: P_j(x)=0\}$, we can choose
an element $a_j\in {\mathbb F_q^2}$ such that the following estimate holds:
$$ \sum_{x\in V_j} V_j(a-x) \lesssim 1 \quad \mbox{for all}~~a\in {\mathbb F_q^2}\setminus \{a_j\},$$
where the constant in the estimate depends only on the degree of $V_j.$
\end{lemma}

\begin{proof} Suppose that $P_j(x)$ is an irreducible polynomial with degree $\geq 2.$ Then,
we aim to prove that there exists a point $a_j\in {\mathbb F_q^2}$ such that the total number of intersection points of the curve $P_j(x)=0$ and the curve $P_j(a-x)=0$ is $\lesssim 1$ for all $a\in {\mathbb F_q^2}\setminus \{a_j\}.$ We claim that it suffices to prove that two curves $P_j(x)=0$ and $P_j(a-x)=0$ are different for all $a \in {\mathbb F_q^2}\setminus \{a_j\}.$ To see this, assume we proved that the curves $P_j(x)=0$ and $P_j(a-x)=0$ are different. Then, $P_j(x)$ and $P_j(a-x)$ can not have a common factor, because the polynomial $P_j(x)$ is irreducible. By Bezout's theorem, we therefore see that the total number of intersection points of two curves can not be greater than the product of the degree of $P_j(x)$ and the degree of $P_j(a-x).$ Thus, it remains to prove that two curves $P_j(x)=0$ and $P_j(a-x)=0$ are not same for all $a\in {\mathbb F_q^2}\setminus \{a_j\}.$
Since $P_j(x)$ is an irreducible polynomial with degree $\geq 2,$ it does not have a linear factor which means that the variety $V_j$ does not contain any line. Without loss of generality, we may assume that there exists $a_j\in {\mathbb F_q^2}$ such that two curves $P_j(x)=0$ and $P_j(a_j-x)=0$ are same. Otherwise, there is nothing to prove. To complete the proof, it is enough to show that if $a\ne a_j,$ then the curve $P_j(a-x)=0$ is not same as the curve $P_j(a_j-x)=0.$ To see this, first note that
for each $\alpha\in {\mathbb F_q^2}$, the graph of $P_j(\alpha-x)=0$ can be obtained by reflecting the graph of $P_j(x)=0$ about the origin and then translating the reflected graph by the vector $\alpha.$ Second, note that the curve given by reflecting the graph of $P_j(x)=0$ about the origin does not contain any line,  because the curve $P_j(x)=0$ does not. Thus, two graphs obtained by shifting the reflected graph by two different vectors can not be same. This completes the proof.
\end{proof}
We now give the complete proof of Theorem \ref{Main}. Suppose that $|V|\sim q$ and  the polynomial $P(x)$ does not have any linear factor.
Assume that the polynomial $P(x)\in \mathbb F_q[x_1,x_2]$ is completely factored by
$$ P(x)=C P_1^{l_1}(x) \cdots P_j^{l_j}(x) \cdots P_n^{l_n}(x),$$
where $P_j(x)$ for each $j=1,\dots,n$ is an irreducible polynomial with the degree $\geq 2.$
For each $j=1,2,\dots,n,$ define the variety $V_j\subset \mathbb F_q^2$ as
$$ V_j=\{x\in \mathbb F_q^2: P_j(x)=0 \},$$
where $P_j(x)$ is an irreducible polynomial with degree at least two.
Then, we see that our variety $V\subset \mathbb F_q^2$ is given by $ V=\cup_{j=1}^n V_j.$
In order to show that $R^*(2\to 4)\lesssim 1, $ we shall show that
\begin{equation}\label{firstaim} \|(fd\sigma)^\vee\|_{L^4({\mathbb F_q^2},dm)}^4 \lesssim \|f\|_{L^2(V,d\sigma)}^4,\end{equation}
for all function $f$ defined on the variety $V.$
Notice from Bezout's Theorem that $|V_i\cap V_j|\sim 1 $ for $i\ne j.$
Thus, given a function $f$ supported on $V$, we may write
$$ f(x)\sim \sum_{j=1}^n f_j(x),$$
where $f_j(x)= f(x)V_j(x)$ and we recall that $V_j(x)$ denotes the characteristic function on the variety $V_j.$
In order to prove the mapping property (\ref{firstaim}), it therefore suffices to show that for every $j=1,\dots,n,$
\begin{equation}\label{secondaim} \|(f_jd\sigma)^\vee\|_{L^4({\mathbb F_q^2},dm)}^4 \lesssim \|f\|_{L^2(V,d\sigma)}^4,\end{equation}
for all function $f$ defined on the variety $V.$
Recall from (\ref{measuredef1}) that the normalized measure $d\sigma$ on $V$ is just a function given by
$$ d\sigma(x)= \frac{q^2}{|V|}V(x).$$
For each $j=1,\dots,n$, define a measure $d\sigma_j$ supported on $V_j$ by
\begin{equation}\label{measuredef1*}d\sigma_j(x)= \frac{q^2}{|V|} V_j(x).\end{equation}
Then, we see that $f_jd\sigma= f_jd\sigma_j.$
From the definition of norms and the Plancherel theorem , we see that for each $j=1,\dots,n,$
\begin{align*} \|(f_jd\sigma)^\vee\|_{L^4({\mathbb F_q^2},dm)}^4=&\|(f_jd\sigma_j)^\vee\|_{L^4({\mathbb F_q^2},dm)}^4\\
=&\|[(f_jd\sigma_j)^\vee]^2\|_{L^2({\mathbb F_q^2},dm)}^2 =\|f_jd\sigma_j\ast f_jd\sigma_j\|_{L^2({\mathbb F_q^2},dx)}^2.\end{align*}

Choose the $a_j\in {\mathbb F_q^2}$ as in Lemma \ref{mainlemma} and write
$$\|f_jd\sigma_j\ast f_jd\sigma_j\|_{L^2({\mathbb F_q^2},dx)}^2= \frac{1}{q^2}\left|f_jd\sigma_j \ast f_jd\sigma_j(a_j)\right|^2
+\frac{1}{q^2} \sum_{x\in {\mathbb F_q^2}\setminus \{a_j\}} \left|f_jd\sigma_j \ast f_jd\sigma_j(x)\right|^2 $$
$$=\mbox{I} +\mbox{II},$$
where we recall that $``dx"$ is the normalized counting measure.
Therefore, our task is to prove that both $\mbox{I}$ and $\mbox{II}$ are $\lesssim \|f\|_{L^2(V,d\sigma)}^4.$
From $(\ref{measuredef1*})$ and  Young's inequality , observe that
\begin{align*} \left|f_jd\sigma_j \ast f_jd\sigma_j(a_j)\right|&\leq \|f_jd\sigma_j \ast f_jd\sigma_j \|_{L^\infty(\mathbb F_q^2, dx)}\leq \frac{q^4}{|V|^2} \|f_j\cdot V_j\|_{L^2({\mathbb F_q^2},dx)}^2\\
&\leq  \frac{q^4}{|V|^2} \|f\cdot V\|_{L^2({\mathbb F_q^2},dx)}^2 =\frac{q^2}{|V|} \|f\|_{L^2(V,d\sigma)}^2.
\end{align*}
Since $|V|\sim q,$ this implies that $ {\mbox{I}}\lesssim \|f\|_{L^2(V,d\sigma)}^4$ as required.
It remains to prove that ${\mbox{II}}\lesssim \|f\|_{L^2(V,d\sigma)}^4.$ Without loss of generality, we may assume that $f\geq 0$ and so $f_j\geq 0.$  By the Cauchy-Schwarz inequality, we see that for every $x\in {\mathbb F_q^2},$
\begin{equation}\label{E1}\left(f_jd\sigma_j \ast f_jd\sigma_j \right)^2(x) \leq (d\sigma_j\ast d\sigma_j)(x) \cdot(f_j^2 d\sigma_j \ast f_j^2 d\sigma_j)(x).\end{equation}
From $(\ref{measuredef1*})$ and the definition of the convolution of functions, observe that for each $x\in {\mathbb F_q^2},$
$$d\sigma_j\ast d\sigma_j(x)= \frac{q^2}{|V|^2}\sum_{y\in V_j} V_j(x-y).$$
By  Lemma \ref{mainlemma}, we therefore see that if $x\in {\mathbb F_q^2}\setminus \{a_j\},$ then
$$ d\sigma_j\ast d\sigma_j(x) \lesssim 1,$$
where we also used the fact that $|V|\sim q.$
Putting this together with $(\ref{E1})$, we obtain that for every $x\in {\mathbb F_q^2}\setminus \{a_j\},$
$$ \left(f_jd\sigma_j \ast f_jd\sigma_j \right)^2(x) \lesssim (f_j^2 d\sigma_j \ast f_j^2 d\sigma_j)(x).$$
Thus, we conclude that
\begin{align*}{\mbox{II}}&= \frac{1}{q^2} \sum_{x\in {\mathbb F_q^2}\setminus \{a_j\}} \left|f_jd\sigma_j \ast f_jd\sigma_j(x)\right|^2\lesssim \frac{1}{q^2} \sum_{x\in {\mathbb F_q^2}}(f_j^2 d\sigma_j \ast f_j^2 d\sigma_j)(x)\\
& =\frac{1}{|V|^2} \left(\sum_{y\in \mathbb F_q^d} f_j^2(y) V_j(y)\right)^2 \leq \|f^2\|^2_{L^1(V,d\sigma)}= \|f\|_{L^2(V,d\sigma)}^4 \end{align*}
where the first equality in the second line follows immediately from Fubini's theorem. This completes the proof of Theorem \ref{Main}.

\end{proof}

\section{Distances between two sets}\label{section4}
In this section, we shall prove Theorem \ref{sharpwolff1} for the Falconer distance problem in two dimensions.

\subsection{ Key estimates for the proof of Theorem \ref{sharpwolff1}}
The proof of Theorem \ref{sharpwolff1} calls for lots of estimates which are related to discrete Fourier analysis.
Here, we collect useful lemmas needed to complete the proof of Theorem \ref{sharpwolff1}.
Given $t\in \mathbb F_q$ and  a polynomial $P(x)\in \mathbb F_q[x_1,x_2]$,  define a variety $V_t$ by
\begin{equation}\label{defV}V_t=\{x\in \mathbb F_q^2: P(x)=t\}.\end{equation}
 Then, we have the following lemma.
\begin{lemma}\label{keylemma1}
If  $P(x)=a_1x_1^d+a_2x_2^d \in \mathbb F_q[x_1,x_2]$ of degree $d\geq 2$, and $a_1, a_2 \in \mathbb F_q\setminus \{0\},$  then we have
$$ \sum_{t\in \mathbb F_q} \widehat{V_t}(m) |V_t| \lesssim 1 \quad \mbox{for all} ~~m\in \mathbb F_q^2\setminus \{(0,0)\}.$$
where $\widehat{V_t}$ is the Fourier transform of the characteristic function on the variety $V_t$  defined by
 $\widehat{V_t}(m) =\frac{1}{q^2} \sum_{x\in \mathbb F_q^2} \chi(-x\cdot m) V_t(x).$
\end{lemma}

Before we proceed to prove Lemma \ref{keylemma1}, we recall some well known facts related to the polynomial $P(x)=a_1x_1^d+a_2x_2^d.$ Lemma 1 in \cite{Todd} says that the polynomial $P(x)-t$ is irreducible for any $t\in \mathbb F_q\setminus \{0\}.$ From Theorem 6.37 in \cite{LN93}, we see that
for every $t\in \mathbb F_q \setminus \{0\},$
\begin{equation}\label{weilsize}
|V_t|=|\{x\in \mathbb F_q^2: P(x)=t\}| = q + O(q^\frac{1}{2}).
\end{equation}
We also recall a theorem by N. Katz (\cite{Ka80}).
\begin{theorem}\label{Katzthm}
Given a polynomial $\Lambda(x) \in \mathbb F_q[x_1,x_2],$ assume that the polynomial $\Lambda(x)$ does not contain a linear factor. Then, for any $m \in \mathbb F_q^2\setminus \{(0,0)\}$ we have
$$\left|\sum_{x \in V}\chi(x\cdot m)\right| \lesssim q^\frac{1}{2},$$
where $V=\{x\in \mathbb F_q^2 : \Lambda(x)=0\}.$
\end{theorem}

\begin{proof} We prove Lemma \ref{keylemma1}.
First, let us write $|V_t| = q + R_t.$
From (\ref{weilsize}),  we have $R_t = O(q^{1/2})$ for $t \ne 0.$
Moreover, it is clear from Schwartz-Zippel lemma that  $R_0 = O(q).$ Using Theorem \ref{Katzthm}, we see that for $t \ne 0$ and $m\ne (0,0),$
 $$\sum_{x \in \mathbb F_q^2 : P(x) = t}  \chi(-x\cdot m) = O(q^\frac{1}{2}).$$
If $t=0$, then we use Schwartz-Zippel lemma to bound
$$\sum_{x \in \mathbb F_q^2 : P(x) = 0}  \chi(-x \cdot m) = O(q).$$
Therefore, we obtain
$$\sum_{t \in \mathbb F_q} |V_t| \sum_{x \in \mathbb F_q^2 : P(x) = t}  \chi(-x\cdot m)$$
$$= q \sum_{t \in \mathbb F_q}  \sum_{x \in \mathbb F_q^2 : P(x) = t}  \chi(-x\cdot m)
+ \sum_{t \in \mathbb F_q} R_t \sum_{x \in \mathbb F_q^2 : P(x) = t}  \chi(-x\cdot m).$$
Now, the first sum vanishes because $m\neq (0,0),$ and the second sum
is written by
$$\sum_{t \ne 0} R_t \sum_{x \in \mathbb F_q^2 : P(x) = t}  \chi(-x \cdot m)+  R_0 \sum_{x \in \mathbb F_q^2 : P(x) = 0}  \chi(-x \cdot m),$$
which is bounded by
$$O(q q^{1/2} q^{1/2} + qq) = O(q^2),$$
which in turn shows that
$$\sum_{t \in \mathbb F_q}\widehat V_t(m)|V_t| \lesssim 1.$$
\end{proof}

In particular, we can take the polynomial $P(x)$ as $\|x\|_2=x_1^2+x_2^2.$
In this case, the variety $V_t$ in (\ref{defV}) is called as a circle with radius $t\in \mathbb F_q$ and we can observe some specific properties on the circle $V_t.$ For example, in \cite{IK10} the Fourier transform on $V_t$ is given by the formula
\begin{equation}\label{explicitcircleF} \widehat{V_t}(m)=q^{-2} \sum_{x\in \mathbb F_q^2} \chi(-m\cdot x) V_t(x)= q^{-1}\delta_0(m) +q^{-3} G_1^2 \sum_{s\in \mathbb F_q\setminus \{0\}}
\chi\left( \frac{\|m\|_2}{4s}+ st\right),\end{equation}
where $\delta_0(m)=1$ if $m=(0,0)$ and $\delta_0(m)=0 $ if $m\ne (0,0),$ and we denote by $G_1$ the usual Gauss sum.
It is known that the Gauss sum $G_1$ is explicitly computed. If $\eta$ is the quadratic character of $\mathbb F_q$ and $\chi$ is the canonical additive character of $\mathbb F_q$, then the Gauss sum $G_1 =\sum_{t\neq 0} \eta(t)\chi(t)$ takes the following value (see Theorem 5.15 in \cite{LN93}):

$$G_1= \left\{\begin{array}{ll}  {(-1)}^{k-1} q^{\frac{1}{2}} \quad &\mbox{if} \quad p =1 \,\,(mod~4) \\
                    {(-1)}^{k-1} i^k q^{\frac{1}{2}} \quad &\mbox{if} \quad p =3 \,\,(mod~4).\end{array}\right., $$
where $k$ is a natural number and $p$ is an odd prime with $q=p^k.$  Thus, if $q=1~~(mod~4)$, then the square of the Gauss sum $G_1$ is exactly $q.$
From this observation and (\ref{explicitcircleF}), we see that if  $q=1~~(mod~4), m\in \mathbb F_q^2,$ and $V_t =\{x\in \mathbb F_q^2: \|x\|_2=t\}$, then
\begin{equation}\label{squarefourier}
\widehat{V_t}(m)=q^{-1}\delta_0(m) +q^{-2} \sum_{s\in \mathbb F_q\setminus \{0\}}
\chi\left( \frac{\|m\|_2}{4s}+ st\right).
\end{equation}
However, observe that $G_1^4$ is always $q^2,$ which yields the following lemma.
\begin{lemma}\label{doubleFourierdecay}
For each $t\in \mathbb F_q,$ let $V_t=\{x\in \mathbb F_q^2: \|x\|_2=t\}.$ For each $m, \xi \in \mathbb F_q^2\setminus \{(0,0)\},$ we have
$$\sum_{t\in \mathbb F_q} \widehat{V_t}(m) \widehat{V_t}(\xi)=q^{-3} \sum_{s\in \mathbb F_q \setminus \{0\}} \chi\left( s(\|m\|_2-\|\xi\|_2)\right).$$
\end{lemma}
\begin{proof} Since $m, \xi \neq (0,0)$ and $G_1^4=q^2,$ the estimate (\ref{explicitcircleF}) implies that
\begin{align*}\sum_{t\in \mathbb F_q} \widehat{V_t}(m) \widehat{V_t}(\xi)=&q^{-4}\sum_{s, s^\prime \in \mathbb F_q\setminus \{0\}}
\chi\left( \frac{\|m\|_2}{4s}+\frac{\|\xi\|_2}{4s^\prime}\right) \sum_{t\in \mathbb F_q} \chi( (s+s^\prime)t)\\
=& q^{-3} \sum_{s \in \mathbb F_q\setminus \{0\}}
\chi\left( \frac{\|m\|_2}{4s}-\frac{\|\xi\|_2}{4s}\right),
\end{align*}
where the last line follows from the orthogonality relation of $\chi.$
Using a change of variables, $1/(4s) \to s$, we complete the proof.
\end{proof}

Let $E, F\subset \mathbb F_q^2.$ We now consider the counting function $\nu: \mathbb F_q \to {\bf N} \cup \{0\},$ given by
$$\nu(t) =|\{(x,y)\in E\times F: \|x-y\|_2=t\}|.$$
In particular, we have the following lemma.
\begin{lemma}\label{zerodistance}
Let $E, F\subset \mathbb F_q^2$. If $|E||F|\gtrsim q^2$ and $q=1~~(mod~4),$  then we have
$$ \nu(0)=O(q^{-1}|E||F|) +q^3 \sum_{m\in V_0} \overline{\widehat{E}}(m) \widehat{F}(m),$$
\end{lemma}
where $V_0=\{x\in \mathbb F_q^2: \|x\|_2=0\}.$
\begin{proof}
It follows that
$$\nu(0)= \sum_{x,y\in \mathbb F_q^2}E(x)F(y) V_0(x-y).$$
Applying the Fourier inversion theorem (\ref{inversionF}) to $V_0(x-y)$ and using the definition of the Fourier transform, we have
$$ \nu(0)=q^4 \sum_{m\in \mathbb F_q^2} \overline{\widehat{E}}(m) \widehat{F}(m) \widehat{V_0}(m).$$
Since $q=1~~(mod ~4)$, the formula (\ref{squarefourier}) can be used to observe the following:
$$ \nu(0) = q^4 \sum_{m\in \mathbb F_q^2} \overline{\widehat{E}}(m) \widehat{F}(m) \left( q^{-1}\delta_0(m) +q^{-2}  \sum_{s\in \mathbb F_q\setminus \{0\}} \chi\left( \frac{\|m\|_2}{4s}\right) \right)$$
$$= q^3 \overline{\widehat{E}}(0,0) \widehat{F}(0,0) + q^2\sum_{m\in \mathbb F_q^2} \overline{\widehat{E}}(m) \widehat{F}(m)\sum_{s\in \mathbb F_q\setminus \{0\}} \chi\left( \frac{\|m\|_2}{4s}\right).$$
Computing the sum over $s\in \mathbb F_q\setminus \{0\}$, it follows that
$$\nu(0)= q^{-1}|E||F| + q^3 \sum_{\|m\|_2=0} \overline{\widehat{E}}(m) \widehat{F}(m) - q^2 \sum_{m\in \mathbb F_q^2} \overline{\widehat{E}}(m) \widehat{F}(m).$$
By the Cauchy-Schwarz inequality and the Plancherel theorem (\ref{plancherel}), the absolute value of the third term above is less than equal to $|E|^{\frac{1}{2}}|F|^{\frac{1}{2}}.$ Since $|E||F|\gtrsim q^2$, the first term dominates the third term and the proof is complete.
\end{proof}

We now address the most important lemma for the proof of Theorem \ref{sharpwolff1}.
The following lemma below may be hard to obtain if we use the direct computation, in part because we do not know the explicit form of the variety.
Using the dual extension theorem, we can overcome the problem.
\begin{lemma}\label{restriction} Let $\Gamma(x)\in \mathbb F_q[x_1,x_2]$ be a non-zero polynomial.
For each $t\in \mathbb F_q$, let $V_t=\{x\in \mathbb F_q^2: \Gamma(x)=t\}.$
Suppose that a set $T\subset \mathbb F_q$ satisfies the following conditions:
if $t\in T$, then $|V_t|\sim q$ and $\Gamma(x)-t$ does not contain a linear factor.
Then, we have that for every set $H\subset \mathbb F_q^2$,
\begin{equation}\label{bigresult} \max_{t\in T} \sum_{m\in V_t} |\widehat{H}(m)|^2 \lesssim q^{-3} |H|^{\frac{3}{2}},\end{equation}
where we recall that $\widehat{H}(m) =q^{-2} \sum_{x\in H} \chi(-x\cdot m)$ and the constant in the estimate depends only on the degree of $V_t$ and on the ratio $|V_t|/q.$
In particular, if $\Gamma(x)=\|x\|_2 =x_1^2+x_2^2,$ then above conclusion holds with $T=\mathbb F_q\setminus \{0\}.$
\end{lemma}
\begin{proof} For every $t\in T$, we must show that
$$q^{-4} \sum_{m\in V_t} \left|\sum_{x\in H} \chi(-x\cdot m)\right|^2 \lesssim q^{-3} |H|^{\frac{3}{2}},$$
where the constant in  ``$\lesssim$" depends only on the degree of $V_t$ and on the ratio $|V_t|/q.$
Since the forms of variables $m, x$ do not affect on above estimate, we can change the variables, $x\leftrightarrow m.$ Thus, it suffices to show that
\begin{equation}\label{majimac}q^{-1}\sum_{x\in V_t} \left|\sum_{m\in H} \chi(-x\cdot m)\right|^2 \lesssim  |H|^{\frac{3}{2}}.\end{equation}
From Theorem \ref{Main}, we see that for every $t\in T,$
$$\|\widehat{fd\sigma}\|_{L^4(\mathbb F_q^2,dm)}\lesssim \|f\|_{L^2(V_t, d\sigma)}$$
for all functions $f$ on $(V_t, d\sigma)$ where $d\sigma$ is the normalized measure on $V_t$ defined as in (\ref{measuredef}).
By duality (see (\ref{defrestriction}) ), this implies that
$$ \|\widehat{g}\|_{L^2(V_t, d\sigma)} \lesssim \|g\|_{L^\frac{4}{3}(\mathbb F_q^2,dm)}$$
for all functions $g$ on $(\mathbb F_q^2,dm).$ If we take $g$ as the characteristic function on the set $H,$ then we have
$$\|\widehat{H}\|^2_{L^2(V_t, d\sigma)} \lesssim \|H\|^2_{L^\frac{4}{3}(\mathbb F_q^2,dm)}.$$
To complete the proof, we shall show that this inequality is same as in (\ref{majimac}).
Namely, it suffices to prove that
\begin{equation}\label{ma1} \|H\|^2_{L^\frac{4}{3}(\mathbb F_q^2,dm)} =  |H|^{\frac{3}{2}}\end{equation}
and
\begin{equation}\label{ma2}
\|\widehat{H}\|^2_{L^2(V_t, d\sigma)}\sim q^{-1}\sum_{x\in V_t} \left|\sum_{m\in H} \chi(-x\cdot m)\right|^2.
\end{equation}
The equality (\ref{ma1}) is clear because $``dm"$ is the counting measure.
To see that (\ref{ma2}) holds, observe from (\ref{measuredef}) that
$$  \|\widehat{H}\|^2_{L^2(V_t, d\sigma)} = \frac{1}{|V_t|} \sum_{x\in V_t} |\widehat{H}(x)|^2.$$
From (\ref{anotherFourier}) observed in Remark \ref{keydef}, we see that the Fourier transform of $H$ takes the following form:
$$\widehat{H}(x)= \sum_{m\in H} \chi(-x\cdot m).$$
\end{proof}
Thus, the statement (\ref{ma2}) follows immediately from the fact that $|V_t|\sim q.$
Thus, the proof of (\ref{bigresult}) is complete. In particular, if $\Gamma(x)=x_1^2+x_2^2,$ then $\Gamma(x)-t$ for $t\ne 0$ is irreducible which implies that the polynomial $\Gamma(x)-t$ for $t\ne 0$ does not have a linear factor. Moreover, $|V_t|\sim q$ for $t\ne 0.$ In this case, we can therefore take $T=\mathbb F_q\setminus \{0\}.$

\subsection{Proof of Theorem \ref{sharpwolff1}}
We shall provide the complete proof of Theorem \ref{sharpwolff1}.
Let $E, F\subset \mathbb F_q^2$ with $|E||F|\gtrsim q^{\frac{8}{3}}.$
For each $x\in \mathbb F_q^2,$ recall that $\|x\|_2=x_1^2+x_2^2.$
We must prove that
\begin{equation}\label{aimclaim}|\Delta(E,F)|=|\{\|x-y\|_2\in \mathbb F_q: x\in E, y\in F\}| \gtrsim q.\end{equation}
For each $t\in \mathbb F_q$, define a variety $V_t=\{x\in \mathbb F_q^2: \|x\|_2=t\}$ and consider a counting function $\nu(t)$ given by
$$\nu(t)=|\{(x,y)\in E\times F: \|x-y\|_2=t\}| = \sum_{x\in E, y\in F} V_t(x-y).$$
Applying the Fourier inversion theorem (\ref{inversionF}) to the function $V_t(x-y) $ and using the definition of the Fourier transform, we see
\begin{align*}\nu(t)=& q^4 \sum_{m\in \mathbb F_q^2}\overline{\widehat{E}}(m) \widehat{F}(m) \widehat{V_t}(m)\\
=&\frac{|E||F||V_t|}{q^2} + q^4 \sum_{m\in \mathbb F_q^2\setminus \{(0,0)\}}\overline{\widehat{E}}(m) \widehat{F}(m) \widehat{V_t}(m).\end{align*}
Squaring the $\nu(t)$ and summing it over $t\in \mathbb F_q$, we obtain
\begin{eqnarray*}\sum_{t\in \mathbb F_q} \nu^2(t)= q^{-4}|E|^2 |F|^2  \sum_{t\in \mathbb F_q} |V_t|^2
+2q^2 |E||F| \sum_{m\in \mathbb F_q^2\setminus \{(0,0)\}}\overline{\widehat{E}}(m) \widehat{F}(m) \sum_{t\in \mathbb F_q} |V_t|\widehat{V_t}(m)\\
+ q^8 \sum_{m, \xi \in \mathbb F_q^2\setminus \{(0,0)\}} \overline{\widehat{E}}(m) \widehat{F}(m) \overline{\widehat{E}}(\xi) \widehat{F}(\xi) \sum_{t\in \mathbb F_q} \widehat{V_t}(m) \widehat{V_t}(\xi)
= \mbox{I} + \mbox{II}+\mbox{III}.\end{eqnarray*}

The Schwartz-Zippel Lemma says that $|V_t|\lesssim q$ for all $t\in \mathbb F_q.$ Therefore, the value $\mbox{I}$ is clearly given by
 $$|\mbox{I}|=O(q^{-1}|E|^2|F|^2).$$

From Lemma \ref{keylemma1} and the Cauchy-Schwarz inequality, the second value can be estimated by
$$|\mbox{II}|\lesssim q^2 |E||F| \left(\sum_{m\in \mathbb F_q^2}\left|\overline{\widehat{E}}(m)\right|^2\right)^{\frac{1}{2}}
\left(\sum_{m\in \mathbb F_q^2}\left|\widehat{F}(m)\right|^2\right)^{\frac{1}{2}}. $$
Applying the Plancherel theorem (\ref{plancherel}) , we obtain
$$|\mbox{II}|=O( |E|^{\frac{3}{2}}|F|^{\frac{3}{2}}).$$
Thus, if $|E||F|\gtrsim q^{\frac{8}{3}}$, then the first term $\mbox{I}$ dominates the second term $\mbox{II}$. It therefore follows that
\begin{equation}\label{D12}
|\mbox{I}| + |\mbox{II}| = O(q^{-1}|E|^2|F|^2).
\end{equation}

We explicitly estimate the third value $\mbox{III}.$
From Lemma \ref{doubleFourierdecay} and the orthogonality relation of the character $\chi$,  observe that

$$\mbox{III}= q^8 \sum_{m, \xi \in \mathbb F_q^2\setminus \{(0,0)\}} \overline{\widehat{E}}(m) \widehat{F}(m) \overline{\widehat{E}}(\xi) \widehat{F}(\xi) \sum_{t\in \mathbb F_q} \widehat{V_t}(m) \widehat{V_t}(\xi)$$
$$ =q^5 \sum_{m, \xi \in \mathbb F_q^2\setminus \{(0,0)\}} \overline{\widehat{E}}(m) \widehat{F}(m) \overline{\widehat{E}}(\xi) \widehat{F}(\xi)
\left( -1+ \sum_{s\in \mathbb F_q } \chi\left( s(\|m\|_2-\|\xi\|_2)\right)\right)$$
$$= -q^5\sum_{m, \xi \in \mathbb F_q^2\setminus \{(0,0)\}} \overline{\widehat{E}}(m) \widehat{F}(m) \overline{\widehat{E}}(\xi) \widehat{F}(\xi)
+ q^6 \sum_{m,\xi \neq (0,0): \|m\|_2=\|\xi\|_2} \overline{\widehat{E}}(m) \widehat{F}(m) \overline{\widehat{E}}(\xi) \widehat{F}(\xi).$$
$$=\mbox{III}_1 + \mbox{III}_2.$$
By the trivial bound and the Cauchy-Schwarz inequality, we have
$$ |\mbox{III}_1| \leq  q^5\left( \sum_{m\in \mathbb F_q^2}   \left|\overline{\widehat{E}}(m)\right| \left|\widehat{F}(m) \right|   \right)^2$$
$$\leq q^5 \left(\sum_{m\in \mathbb F_q^2} \left|\overline{\widehat{E}}(m)\right|^2\right) \left( \sum_{m\in \mathbb F_q^2} \left|\widehat{F}(m) \right|^2\right).$$ Applying the Plancherel theorem (\ref{plancherel}), the value $\mbox{III}_1$ is estimated by
$$ |\mbox{III}_1|=O(q|E||F|).$$

To estimate the value $\mbox{III}_2$, observe that
$$ \mbox{III}_2 = q^6 \sum_{k\in \mathbb F_q} \left( \sum_{m\ne (0,0): \|m\|_2=k\}} \overline{\widehat{E}}(m) \widehat{F}(m)\right)^2$$
$$= q^6 \left( \sum_{m\ne (0,0): \|m\|_2=0} \overline{\widehat{E}}(m) \widehat{F}(m)\right)^2 + q^6 \sum_{k\in \mathbb F_q\setminus \{0\}} \left( \sum_{\|m\|_2=k} \overline{\widehat{E}}(m) \widehat{F}(m)\right)^2$$
$$= q^6\left( \sum_{m\in V_0} \overline{\widehat{E}}(m) \widehat{F}(m) -\overline{\widehat{E}}(0,0) \widehat{F}(0,0) \right)^2 + q^6 \sum_{k\in \mathbb F_q\setminus \{0\}} \left( \sum_{m\in V_k} \overline{\widehat{E}}(m) \widehat{F}(m)\right)^2.$$
Since $\overline{\widehat{E}}(0,0) \widehat{F}(0,0)=q^{-4}|E||F|$, expanding the first term above and putting it together with the second term, we have
$$\mbox{III}_2 =q^6 \sum_{k\in \mathbb F_q} \left( \sum_{m\in V_k} \overline{\widehat{E}}(m) \widehat{F}(m)\right)^2 -2q^2 |E||F|
 \sum_{m\in V_0} \overline{\widehat{E}}(m) \widehat{F}(m) +  q^{-2}|E|^2 |F|^2 .$$
Putting this estimate together with the estimate $(\ref{D12})$,  we obtain
$$ \sum_{t\in \mathbb F_q} \nu^2(t) = q^6 \sum_{k\in \mathbb F_q} \left( \sum_{m\in V_k} \overline{\widehat{E}}(m) \widehat{F}(m)\right)^2 -2q^2 |E||F|
 \sum_{m\in V_0} \overline{\widehat{E}}(m) \widehat{F}(m) + O(q^{-1} |E|^2|F|^2).$$
Observe that the absolute value of the second term above is less than equal to the number
$$ 2q^2 |E||F|  \sum_{m\in \mathbb F_q^2} \left|\overline{\widehat{E}}(m)\right| \left|\widehat{F}(m)\right|.$$
Using the Cauchy-Schwarz inequality and the Plancherel theorem (\ref{plancherel}), this value is dominated by  $2|E|^\frac{3}{2} |F|^\frac{3}{2}.$
Since we have assumed that $|E||F|\gtrsim q^\frac{8}{3}$, the third term  dominates the second term and so we obtain that
\begin{equation}\label{finallong}
\sum_{t\in \mathbb F_q} \nu^2(t) = q^6 \sum_{k\in \mathbb F_q} \left( \sum_{m\in V_k} \overline{\widehat{E}}(m) \widehat{F}(m)\right)^2 + O(q^{-1}|E|^2|F|^2).
\end{equation}
We are ready to prove that the statement (\ref{aimclaim}) holds.
First we assume that $q=3~~(mod~4).$
In this case, $-1$ is not a square number in $\mathbb F_q,$ because $\eta(-1)=-1$ where $\eta$ is the quadratic character of $\mathbb F_q$ (see Remark $5.13$ in \cite{LN93}). Thus, we see that $V_0=\{x\in \mathbb F_q^2: \|x\|_2=0\}=\{(0,0)\}.$
Therefore, we see that
\begin{equation}\label{goodone}\sum_{t\in \mathbb F_q} \nu^2(t) = q^6 \left(\overline{\widehat{E}}(0,0) \widehat{F}(0,0)\right)^2 + q^6 \sum_{k\in \mathbb F_q\setminus \{0\}} \left( \sum_{m\in V_k} \overline{\widehat{E}}(m) \widehat{F}(m)\right)^2 + O(q^{-1}|E|^2|F|^2).\end{equation}
Note that the first term is same as $ q^{-2}|E|^2|F|^2$ which is dominated by the third term.
In addition, observe that  the absolute value of the second term is less than equal to the value
$$ q^6 \left(\max_{k\in \mathbb F_q\setminus \{0\}}\left|\sum_{m\in V_k} \overline{\widehat{E}}(m) \widehat{F}(m)\right|\right)
 \left( \sum_{m\in \mathbb F_q^2} \left|\overline{\widehat{E}}(m)\right| \left|\widehat{F}(m)\right|\right).$$
In order to get the upper bound of the maximum value, we use the Cauchy-Schwarz inequality and apply Lemma \ref{restriction}. Then, we see
$$\left(\max_{k\in \mathbb F_q\setminus \{0\}}\left|\sum_{m\in V_k} \overline{\widehat{E}}(m) \widehat{F}(m)\right|\right) \lesssim
q^{-3} |E|^\frac{3}{4} |F|^\frac{3}{4}.$$ On the other hand, using the Cauchy-Schwarz inequality and the Plancherel theorem (\ref{plancherel}) yield that
$$\left( \sum_{m\in \mathbb F_q^2} \left|\overline{\widehat{E}}(m)\right| \left|\widehat{F}(m)\right|\right) \leq  q^{-2} |E|^\frac{1}{2} |F|^\frac{1}{2}.$$
Therefore, the second term in (\ref{goodone}) can be estimated by
\begin{equation}\label{nicework}\left|q^6 \sum_{k\in \mathbb F_q\setminus \{0\}} \left( \sum_{m\in V_k} \overline{\widehat{E}}(m) \widehat{F}(m)\right)^2\right| \lesssim q|E|^\frac{5}{4} |F|^\frac{5}{4}.\end{equation}
Putting together all estimates , we see that
$$\sum_{t\in \mathbb F_q} \nu^2(t) \lesssim q|E|^\frac{5}{4} |F|^\frac{5}{4} + q^{-1}|E|^2|F|^2 .$$
Since $|E||F|\gtrsim q^\frac{8}{3},$ it follows that
$$\sum_{t\in \mathbb F_q} \nu^2(t) \lesssim  q^{-1}|E|^2|F|^2 .$$
It is clear that
$$ (|E||F|)^2= \left(\sum_{t\in \Delta(E,F)} \nu(t)\right)^2 \leq |\Delta(E,F)| \sum_{t\in \mathbb F_q} \nu^2(t),$$
where we used the Cauchy-Schwarz inequality .
Thus, we have proved that if $q=3~~(mod~4)$ and $|E||F|\gtrsim q^\frac{8}{3},$ then
$$ |\Delta(E,F)| \gtrsim q.$$
It remains to prove that if $ q=1~~(mod~4)$ and $|E||F|\gtrsim q^\frac{8}{3},$ then  the estimate (\ref{aimclaim}) holds.
Assume that $ q=1~~(mod~4).$ Since $|E||F|=\sum_{t\in \Delta(E,F)} \nu(t)$, it follows that
\begin{equation}\label{idea} \left(|E||F| -\nu(0)\right)^2 = \left(\sum_{t\in \Delta(E,F)\setminus \{0\}} \nu(t)\right)^2 \leq |\Delta(E,F)| \sum_{t\in \mathbb F_q\setminus \{0\}} \nu^2(t).\end{equation}
From Lemma \ref{zerodistance}, recall that
\begin{equation}\label{v0}\nu(0)=O(q^{-1}|E||F|) +q^3 \sum_{m\in V_0} \overline{\widehat{E}}(m) \widehat{F}(m).\end{equation}
Thus, using the estimate (\ref{finallong}) yields
$$ \sum_{t\in \mathbb F_q\setminus \{0\}} \nu^2(t)=\sum_{t\in \mathbb F_q} \nu^2(t) - \nu^2(0)$$
$$= q^6 \sum_{k\in \mathbb F_q\setminus \{0\}} \left( \sum_{m\in V_k} \overline{\widehat{E}}(m) \widehat{F}(m)\right)^2 +  O(q^2|E||F|) \left( \sum_{m\in V_0} \overline{\widehat{E}}(m) \widehat{F}(m)\right)+
O(q^{-1}|E|^2|F|^2) $$
As in (\ref{nicework}), the absolute value of the first term is $\lesssim q|E|^\frac{5}{4} |F|^\frac{5}{4}$ which is dominated by the third term, because we have assumed that $|E||F|\gtrsim q^\frac{8}{3}.$ To estimate the absolute value of the second term, notice that
\begin{equation}\label{almost}\left| \sum_{m\in V_0} \overline{\widehat{E}}(m) \widehat{F}(m)\right|\leq  \sum_{m\in \mathbb F_q^2} \left|\overline{\widehat{E}}(m)\right| \left|\widehat{F}(m)\right| \leq q^{-2} |E|^\frac{1}{2} |F|^\frac{1}{2}, \end{equation}
where the last inequality can be obtained by using the Cauchy-Schwarz inequality and the Plancherel theorem (\ref{plancherel}).
Thus, the absolute value of the second term  is $\lesssim |E|^\frac{3}{2} |F|^\frac{3}{2},$ which is also dominated by the third term if $|E||F|\gtrsim q^\frac{8}{3}.$ Thus, we obtain that
\begin{equation}\label{bottom}
\sum_{t\in \mathbb F_q\setminus \{0\}} \nu^2(t) \lesssim q^{-1}|E|^2|F|^2.
\end{equation}
Next, we claim that
\begin{equation}\label{top} \left(|E||F| -\nu(0)\right)^2 \sim  |E|^2 |F|^2.
\end{equation}
To see this,   observe from (\ref{v0}) and (\ref{almost}) that
$$ |\nu(0)|\lesssim q^{-1}|E||F| +q|E|^\frac{1}{2} |F|^\frac{1}{2} \sim  q|E|^\frac{1}{2} |F|^\frac{1}{2},$$
where the last estimate is clear because $|E||F| \leq q^4.$
Since $|E||F|\gtrsim q^\frac{8}{3}$, it therefore is clear that $|E||F|$ dominates $|\nu(0)|$ and the estimate (\ref{top}) holds.
From (\ref{idea}), (\ref{bottom}), and (\ref{top}), we conclude that
if $q=1 ~~(mod~4)$ and $|E||F|\gtrsim q^\frac{8}{3},$ then
$$ |\Delta(E,F)| \gtrsim q.$$
This completes the proof of Theorem \ref{sharpwolff1}.

\end{document}